\documentclass[12pt]{amsart}
\pdfoutput=1
\usepackage[utf8]{inputenc}
\usepackage[margin=1in]{geometry}
\usepackage{verbatim}

\usepackage{hyperref}
\usepackage{amsfonts}
\usepackage{amsmath}
\usepackage{amsthm}
\usepackage{amssymb}
\usepackage{graphicx}
\usepackage{enumerate}
\usepackage[usenames, dvipsnames]{color}
\usepackage[dvipsnames]{xcolor}
\usepackage{mathrsfs}
\usepackage{ytableau}
\usepackage{newtxtext, newtxmath}
\usepackage{blkarray,booktabs,bigstrut}
\usepackage{mathtools,tabstackengine}

\TABstackMath
\setstacktabbedgap{10pt}
\setstackgap{L}{1.5\baselineskip}
\newcommand\cddots{\raisebox{-1pt}{$\ddots$}}

\newtheorem{theorem}{Theorem}[section]
\newtheorem*{theorem*}{Theorem}
\newtheorem{lemma}[theorem]{Lemma}

\newtheorem{corollary}[theorem]{Corollary}
\theoremstyle{definition}
\newtheorem{definition}[theorem]{Definition}
\newtheorem{notation ex}[theorem]{Notation and Example}
\newtheorem{def rem}[theorem]{Definition and Remark}
\newtheorem{example}[theorem]{Example}

\newtheorem{remark}[theorem]{Remark}

\newcommand{\qcol}{\textcolor{WildStrawberry}}

\begin{document}

\title[SLP of modules over Clements-Lindstr{\"o}m rings]{The strong Lefschetz property of certain modules over Clements-Lindstr{\"o}m rings}
\author{Bek C. Chase}

\begin{abstract}
    We introduce a method for studying the Lefschetz properties for $k[x,y]$-modules based on the Lindstr{\"o}m-Gessel-Viennot Lemma. In particular, we prove that certain modules over Artinian Clements-Lindstr{\"o}m rings in characteristic zero have the strong Lefschetz property. In particular, we show that every homogeneous idea in a Clements-Lindstr{\"o}m ring of embedding dimension two has the strong Lefschetz property. As an application, we study the strong Lefschetz property of type two monomial ideals of codimension three. 
\end{abstract}
\maketitle

\section{Introduction} 
The Lefschetz properties are an active area of research in commutative algebra. Much work has been done in studying when Artinian graded $k$-algebras have the weak or strong Lefschetz properties (WLP or SLP); see \cite{survey} or \cite{book} for an overview. Less is known about the Lefschetz properties for graded modules, even in few variables. Some results have been obtained about the weak Lefschetz property for modules in two variables in \cite{modulepaper1} and \cite{modulepaper2}, but much is still unknown.


The difficulty in studying the Lefschetz properties for modules is that many of the methods used to study them for algebras do not easily apply to the module case. A useful criterion for checking the presence of the WLP and SLP in an Artinian $k$-algebra, given by Wiebe in \cite{wiebe}, is to check it for the generic initial ideal in degree reverse lexicographic order, which we denote $\text{gin}(\text{---})$: for a homogeneous ideal $J \subset S = k[x_1,\dots,x_n]$, where $k$ is an infinite field, $S/J$ has the WLP (respectively, SLP) if and only $S/\text{gin}(J)$ has the WLP (respectively, SLP). For the same to be true for modules over $S$, we would need to consider presentations of these modules before taking the generic initial ideal, which would often be impractical. If the module is a quotient $I/J$ of ideals $J \subset I$ in $S$, one might hope instead that $I/J$ has the WLP or SLP as an $S$-module if and only if $\text{gin}(I)/\text{gin}(J)$ has the WLP or SLP as an $S$-module. This is, inconveniently, not true; see Example \ref{lexex}. Another method is required.

In this paper, we consider modules of the form $I/J$, where $J \subset I$ are ideals in $S = k[x,y]$ and $J$ is an Artinian monomial complete intersection ideal, making $k[x,y]/J$ a Clements-Lindstr{\"o}m ring which itself has the strong Lefschetz property. We first review the definitions of the strong and weak Lefschetz properties in Section \ref{prelimsect} and other important facts, as well as give a few motivating examples.

Given ideals $I$ and $J$ as above, in Section \ref{mainsect} we prove our main result, that the $k[x,y]$-module $I/J$ has the SLP in characteristic zero. The key to the proof is a corollary of the Lindstr{\"o}m-Gessel-Viennot Lemma on the enumeration of non-intersecting lattice paths, used to show that the matrix of binomial coefficients representing multiplication by a linear form has maximal rank. In the last section we show that knowing such modules have the SLP can be useful in proving that some Artinian $k$-algebras of small type have the strong Lefschetz property. Essential to this application is the fact that the \textit{central simple modules} (the definition of which, owing to Harima and Watanbe, we review in Section \ref{prelimsect}) of such algebras can take the form of the modules we study in Section \ref{mainsect}. 

Throughout this paper we assume that $k$ is a field of characteristic zero.

\section{Preliminaries}\label{prelimsect}

\begin{definition}
Let $A = \bigoplus_{i=0}^c A_i$ be a standard graded Artinian $k$-algebra where $k$ is a field of characteristic zero. If $M = \bigoplus_{i=p}^q M_i$, where $M_p \neq 0$ and $M_q \neq 0$, is a finite graded $A$-module, then the Hilbert function of $M$ is the map $i \mapsto h_i := \text{dim}_k M_i$ and the Hilbert series of $M$ is the polynomial
\[
H_M(t) = \sum_i h_i t^i.
\]
If the Hilbert series of $M$ is symmetric, that is, $\text{dim}_k M_{p+i} = \text{dim}_k M_{q-i}$ for all $i$, then we call $(p+q)/2$ the \textit{reflecting degree} of $H_M(t)$, or just the reflecting degree of $M$ for convenience. If $M$ and $M'$ are two finite $A$-modules, then we say the reflecting degrees $r$ of $M$ and $r'$ of $M'$ \textit{coincide} if $r = r'$ or $r \pm 1/2 = r'$.
\end{definition}

\begin{definition}\label{LPdefn}
We say that $M = \bigoplus_{i=p}^q M_i$, a finite graded $A$-module, has the \textit{weak Lefschetz property} (WLP) if there exists a linear form $\ell \in A_1$ such that the multiplication map $\times \ell\colon M_i \rightarrow M_{i+1}$ has maximal rank for all $i$, i.e., for each $i$ it is either injective or surjective. In this case we call $\ell$ a \textit{weak Lefschetz element}. Similarly, we say that $M$ has the \textit{strong Lefschetz property} (SLP) if there exists a linear form $\ell \in A_1$ such that the multiplication map $\times \ell^d \colon M_i \rightarrow M_{i+d}$ has maximal rank for all $d>0$ and all $i$, and call such a linear form a \textit{strong Lefschetz element}. If the Hilbert series of $M$ is symmetric, then it is common to say that $M$ has the \textit{strong Lefschetz property in the narrow sense}.
\end{definition}

A useful implication of a module having the strong Lefschetz property is that its Hilbert series is unimodal (Remark 2.2, \cite{nonstd}). Just as for algebras, the converse of this is false; unlike for algebras, even in one variable modules with a unimodal Hilbert series do not necessarily have the WLP, let alone the SLP. This is evidenced by the following example.

\begin{example}\label{1varex}
Consider the $k[x]$-module $M = \frac{k[x]}{(x^2)} \oplus \frac{k[x]}{(x^2)} \left( -2 \right) $
where the grading of the second summand is shifted 2 degrees. This module has the unimodal and symmetric Hilbert series $H_M(t) = 1+t+t^2+t^3$, and each summand itself has the SLP. Yet $M$ does not have the WLP, as can be seen by looking at the multiplication map $\times (x): M_1 \to M_2$ which is neither injective nor surjective. 
\end{example}

Nonetheless, the situation over $k[x]$ is not difficult to understand, since there is only one linear form, $x$, and pathological examples are similar to the above and easy to describe using the structure theorem for finitely generated modules over a PID. The two variable case, however, merits further study, as tools used to prove algebras have the WLP or SLP can fail when we attempt to use them for modules, as illustrated in the next example.

\begin{example}\label{lexex}
    Let $I = (x^3,y^4)$ and $J = (x^5,y^5)$ be ideals in $k[x,y]$ and let $I/J$ be viewed as a $k[x,y]$-module. If we consider $G := \text{gin}(I)/\text{gin}(J)$, we will see that our naive form of Wiebe's criterion mentioned in the introduction fails here. In two variables, the generic initial ideal is a lexicographic ideal, so we may assume that $\text{gin}(I) = \text{Lex}(I) = (x^3, x^2y^2, xy^4,y^6)$ and $\text{gin}(J) = \text{Lex}(J) = (x^5,x^4y, x^3y^3, x^2y^5, xy^7,y^9)$. Because these are monomial ideals, it is enough to check if $x+y$ is a weak Lefschetz element. In degrees 4 and 5, $G$ has $k$-dimension 3, so we would expect $\times (x+y) \colon G_4 \to G_5$ to be bijective if indeed $G$ has the WLP. However, it has a nonzero kernel since $x^4(x+y) = 0$, hence $G$ has neither the WLP nor the SLP. On the other hand, $I/J$ does have the SLP as an $S$-module, which can be checked directly by hand or by computer. Furthermore, we prove in Section \ref{mainsect} that $I/J$ belongs to a class of modules which all have the strong Lefschetz property.
\end{example}

Our tool for the module setting is the Lindstr{\"o}m-Gessel-Viennot Lemma, originally given by Lindstr{\"o}m in 1973 (\cite{lindstrom}, Lemma 1) and later proved independently by Gessel and Viennot in 1989 (\cite{GV}, Theorem 1). This lemma gives a combinatorial interpretation for the determinant of a matrix of binomial coefficients as a count of non-intersecting lattice paths. The following corollary by Gessel and Viennot (\cite{GVcor}) is a consequence of the lemma. Note that for our purposes, we say that $\binom{a_i}{b_i} = 0$ if $b_i <0$, so in Corollary \ref{gvcor} we have $\binom{a_i}{b_i} \neq 0$ if and only if $0 \leq b_i \leq a_i$.  

\begin{corollary}[Corollary 2, \cite{GVcor}]\label{gvcor}
    Let $a_1 < a_2 < \dots <a_m$ and $b_1 < b_2 < \dots < b_m$ be two sets of integers. Then the determinant of the matrix $M = (m_{ij})$ of binomial coefficients $m_{ij} = \binom{a_i}{b_j}$ is nonnegative, and is positive if and only if $\binom{a_i}{b_i} \neq 0$ for each $i$. 
\end{corollary}

We now recall the notion of central simple modules (CSMs) and the characterization of the strong Lefschetz property in terms of CSMs. For more details, refer to \cite{central-simple}, \cite{nonstd}, and \cite{commutator}.

\begin{definition}
    Let $A$ be a standard graded Artinian $k$-algebra and $\ell \in A_1$ a linear form with $r$ the smallest positive integer for which $\ell^r = 0$. Then we have the following sequence 
    \[
    A = (0:\ell^r)+(\ell)  \supset (0:\ell^{r-1}) +(\ell) \supset \dots \supset (0:\ell)+(\ell) \supset (0:\ell^0)+(\ell) = (\ell).
    \]
    The $i$th \textit{ central simple module} $V_{i,\ell}$ of $A$ with respect to the linear form $\ell$ is defined as the $i$th nonzero successive quotient of this sequence, so it is of the form
    \[
    V_{i,\ell} = \frac{(0:\ell^{f_i})+(\ell)}{(0:\ell^{f_{i}-1})+(\ell)}
    \]   
    where $r \geq f_i > f_{i+1} \geq 1$ for all $i$. Note that $V_{1,\ell} = A/{(0:\ell^{r-1})+(\ell)}$. We also define $\widetilde{V}_{i,\ell} = V_{i,\ell} \otimes k[t]/(t^{f_i})$ and $\widetilde{V} = \bigoplus_i \widetilde{V}_{i,\ell}$. Then, viewing $\widetilde{V}_{i,\ell}$ and $K[t]/(t^{f_i})$ as $k$-vector spaces, we have \[H_{\widetilde{V}_{i,\ell}}(t) = H_{V_{i,\ell}}(t)(1+t+\dots + t^{f_i -1})\]
    which is symmetric if and only if $H_{V_{i,\ell}}(t)$ is symmetric (Remark 2.2(1), \cite{nonstd}).
\end{definition}

\begin{theorem}[Theorem 5.2, \cite{nonstd}]\label{hw-csm-thm}
Let $A$ be a graded Artinian $k$-algebra with a symmetric Hilbert series. Then the following are equivalent 
\begin{enumerate}
    \item $A$ has the SLP in the narrow sense. 
    \item There exists a linear form $\ell \in A_1$ such that $V_{i,\ell}$ has the SLP in the narrow sense and the reflecting degree of $\widetilde{V}_{i,\ell}$ coincides with that of $A$ for all $i$. 
\end{enumerate} 
\end{theorem}

In the situation that $A$ does not have a symmetric Hilbert series, we still have the following weaker version of Theorem \ref{hw-csm-thm}.

\begin{theorem}\label{csm-nonsym} 
Let $A$ be a graded Artinian $k$-algebra. Then $A$ has the SLP if there exists a linear form $\ell \in A_1$ such that $\widetilde{V}$ has the SLP.  
\end{theorem}

Theorem \ref{csm-nonsym} is a consequence of the fact that, regardless of whether or not $A$ has a symmetric Hilbert series, the associated graded ring $Gr_{(l)}(A)$, which we define below, has the SLP if $\widetilde{V}$ does. For a proof of this, we refer the reader to Theorem 2.7 in \cite{nonstd}, noting that the relevant portion of the proof does not require the Hilbert series to be symmetric. The result then follows from Theorem \ref{assgr} below. Note that our Theorem \ref{assgr} is a generalization of Theorem 4.6 in \cite{nonstd} and has a similar proof.

\begin{remark}\label{directsumlp}
    Suppose $M$ and $N$ are modules over a ring $R$ which have the same strong Lefschetz element $\ell$. Then $\ell$ is a strong Lefschetz element for $M \oplus N$ if the multiplication maps $\times \ell^d \colon M_i \to M_{i+d}$ and $\times \ell^d \colon N_i \to N_{i+d}$ are either both injective or both surjective for all $i$ and for all $d$. In particular, if the Hilbert series of $M$ and of $N$ are both symmetric, $M \oplus N$ has the strong Lefschetz property if and only if the reflecting degrees of $M$ and $N$ coincide.
\end{remark}

\begin{definition}
    Let $A = \bigoplus_{i=0}^c A_i$ be an Artinian graded $k$-algebra and let $l\in A_1$ with $r$ the least integer such that $l^r = 0$. Recall that the \textit{associated graded ring} of $A$ with respect to $l$ is defined as 
    $$Gr_{l}(A) = A/(l) \oplus (l)/(l^2) \oplus (l^2)/(l^3) \oplus\dots\oplus (l^{r-1})/(l^r).$$
\end{definition}

We also need the following lemma. 

\begin{lemma}[Remark 3.1.7, \cite{lindsey-thesis}]\label{r/i tensor}
Let $I$ be a homogeneous ideal of $R$. Then $R/I$ has the strong Lefschetz property if and only if $R/I \otimes_k k[z]/(z^{c})$ has the weak Lefschetz property for all $c \geq 0$.
\end{lemma}

\begin{theorem}\label{assgr}
    Let $A$ be a graded Artinian $k$-algebra. Then A has the WLP (respectively SLP) if and only if $Gr_{l}(A)$ has the WLP (respectively SLP) for some linear form $l \in A$. 
\end{theorem}
\begin{proof}
    For details of the proof for the WLP and the forward implication, see  Theorem 4.6 of \cite{nonstd}, noting that again it is not required for $A$ to have a symmetric Hilbert series in that portion of the proof. Here we prove only that if $Gr_{l}(A)$ has the SLP then $A$ has the SLP.

    Let $t$ be an indeterminate and define $\widetilde{A} = A \otimes_k k[z]/(z^{c})$ where $c$ is any positive integer. Then we have $$Gr_{l}(\widetilde{A}) \cong Gr_{l}(A) \otimes_k k[z]/(z^{c})$$
    thus it follows from Lemma \ref{r/i tensor} that $Gr_{l}(\widetilde{A})$ has the WLP, hence $\widetilde{A}$ has the WLP. Since this holds for any $c$, $A$ has the SLP, again by Lemma \ref{r/i tensor}. 
\end{proof}

\section{A class of modules with the SLP in 2 variables}\label{mainsect}

In this section we prove our main results pertaining to the strong Lefschetz property for $k[x,y]$-modules.

\begin{theorem}\label{mainmodulethm}
    Let $J = (x^a,y^b)$ be an ideal in $S = k[x,y]$ and $I$ a monomial ideal in $S$ such that $J \subset I$. Then $I/J$ has the SLP as an $S$-module.
\end{theorem}

\begin{proof}
    Let $\ell = x+y$. We will show that $\ell$ is a strong Lefschetz element for the nonzero module $M := I/J$. Consider the multiplication map $\times \ell^d \colon M_i \to M_{i+d}$ for some fixed $i$ and $d$. We may assume that $i > 0$, since either $\text{dim}_k M_0 = 0$ and the map trivially has maximal rank, or $\text{dim}_k M_0 = 1$ and $M = S/J$ which is already known to have the SLP. Similarly, we assume $i< a+b-2$, where $a+b-2$ is the socle degree of both $S/J$ and $M$, else $\text{dim}_k M_{i+d} = 0$. In fact, we  may assume that $1 < i+d \leq a+b-2$ so that $\text{dim}_k M_{i+d} \neq 0$. 
    
    Now for any such $i$ and $d$, the multiplication map $\times \ell^d \colon S_i \to S_{i+d}$ can be represented by a matrix of binomial coefficients; for convenience, and because we are concerned only with its rank, we will work with the transposed matrix 
    \begin{equation*}
    \begin{blockarray}{ccccccccc}
    & x^{i+d} & x^{i+d-1}y & x^{i+d-2}y^2 &\dots &x^iy^d &x^{i-1}y^{d+1} &\dots & y^{i+d} \\[6pt]
    \begin{block}{c[cccccccc]}
    x^i & \binom{d}{0} & \binom{d}{1} & \binom{d}{2} & \dots & \binom{d}{d} & 0& \dots &0\\[6pt]
    x^{i-1}y & 0& \binom{d}{0}& \binom{d}{1} & \dots & \binom{d}{d-1} & \binom{d}{d} &\dots &0 \\[6pt]
    \vdots & \vdots &\vdots &\vdots  &\vdots  &\vdots  &\vdots  &\vdots  &\vdots \\[6pt]
    y^i & 0 & 0 & 0 &\dots &\binom{d}{d-i} &\binom{d}{d-i+1} &\dots &\binom{d}{d} \\[6pt]
    \end{block}
    \end{blockarray} 
    = A.
    \end{equation*}
    This matrix has maximal rank because $\ell$ is (trivially) a strong Lefschetz element for $S$. Now consider the matrix $A'$ representing the multiplication map $\times \ell^d \colon [S/J]_i \to [S/J]_{i+d}$, which can be derived from $A$ by first deleting the first $m_1 = \text{max}\{i+d-a+1, 0\}$ columns and the last $m_2 = \text{max}\{i+d-b+1,0\}$ columns to get the following matrix:
    
    \begin{equation*}
    \begin{blockarray}{cccccccc}
    & x^{i+d-m_1}y^{m_1} & x^{i+d-m_1-1}y^{m_1+1} & \dots & x^iy^d & x^{i-1}y^{d+1}& \dots & x^{m_2}y^{i+d-m_2} \\[6pt]
    \begin{block}{c[ccccccc]}
    x^i & \binom{d}{m_1} & \binom{d}{m_1+1} & \dots & \binom{d}{d} & \binom{d}{d+1} & \dots & \binom{d}{i+d-m_2} \\[6pt]
    x^{i-1}y & \binom{d}{m_1-1} & \binom{d}{m_1} & \dots & \binom{d}{d-1} & \binom{d}{d} & \dots & \binom{d}{i+d-m_2-1} \\[6pt]
    \vdots & \vdots  & \vdots  & \vdots  & \vdots  & \vdots  & \vdots  & \vdots \\[6pt]
    y^i & \binom{d}{m_1-i} & \binom{d}{m_1-i+1} & \dots & \binom{d}{d-i} & \binom{d}{d-i+1} & \dots & \binom{d}{d-m_2} \\[6pt]
    \end{block}    
    \end{blockarray}
    \end{equation*}
    Then deleting certain rows of the above matrix if the corresponding monomials are in $J$ gives us $A'$. Notice that some of the entries of $A'$ may be zero, since $\binom{d}{k} = 0$ if $k<0$ or $k>d$. 
    
    The matrix $A'$ also has maximal rank since $S/J$ is an Artinian monomial algebra in two variables, so $\ell$ is a strong Lefschetz element for $S/J$. In fact, $A'$ has at least one nonzero element in each row regardless of the size of the matrix. This follows from the fact that $S/J$ is Gorenstein. To show this directly, suppose $x^r y^s \not\in J$ is a nonzero monomial labeling a row of $A'$.  Consider $(x^r y^s)(x^{a-1-r} y^{b-1-s}) = x^{a-1}y^{b-1}$, which is nonzero since $x^{a-1}y^{b-1}$ is the socle of $S/J$. Because of our assumptions on the size of $d$ and $i+d$, any degree $d$ monomial factor of $x^{a-1-r} y^{b-1-s}$ multiplied by $x^ry^s$ corresponds to a column with a nonzero entry in the row in question.


    Now we can get the matrix $B$ representing the map $\times \ell^d\colon M_i \to M_{i+d}$ by deleting rows of the matrix for $S/J$ corresponding to monomials of $S_i$ which are not in $I$, i.e., rows corresponding to monomial labels $x^s y^r$ where, for every minimal monomial generator $x^py^q$ of $I$, $s<p$ or $r<q$. Note that there is no need to delete columns indexed by monomials not in $I$, since $I$ is an ideal so the corresponding matrix entry is already zero. 

    The $n$th row of the matrix $B$ can be written as 
\[
r_n =  
\begin{bmatrix}
    \binom{d}{k_n} & \binom{d}{k_n+1} & \dots & \binom{d}{k_n + c-1}
\end{bmatrix}
\]
where $c$ is the number of columns of $B$, $m_1 - i \leq k_n \leq m_1$, and $k_i > k_j$ for $i<j$. We also have $k_n+c-1 \geq 0$, because each row of $A'$ and therefore $B$ has at least one nonzero entry, so at least one of $k_n, k_n+1, \dots, k_n+c-1$ must be nonnegative, and $k_n+c-1$ is the largest. Similarly, at least one of the binomial coefficients must be nonzero, so at least we must have $d \geq k_n$.

Now we use a series of column operations relying on the relation between binomial coefficients 
\[
\binom{n}{k} = \binom{n-1}{k-1} + \binom{n-1}{k}.
\]
We first add the second column to the first, so the $n$th row becomes
\[ 
\begin{bmatrix}
    \binom{d+1}{k_n+1} & \binom{d}{k_n+1} & \dots & \binom{d}{k_n + c-1}
\end{bmatrix}.
\]
We then add the third column to the second, then the second to the first again, and so on until we go through this process starting with the last column. The result of these column operations is the matrix $B'$ with rows of the form 
\[
r_n' =  
\begin{bmatrix}
    \binom{d+c-1}{k_n+c-1} & \binom{d+c-2}{k_n+c-1} & \dots & \binom{d}{k_n + c-1}
\end{bmatrix}
\]

i.e., the matrix

\begin{equation*}
B' =
    \begin{bmatrix}
        \binom{d+c-1}{k_1+c-1} & \binom{d+c-2}{k_1+c-1} & \dots & \binom{d}{k_1+c-1} \\[6pt]
        \binom{d+c-1}{k_2+c-1} & \binom{d+c-2}{k_2+c-1} & \dots & \binom{d}{k_2+c-1} \\[6pt]
        \vdots & \vdots &\vdots &\vdots \\[6pt]
        \binom{d+c-1}{k_s+c-1} &\binom{d+c-2}{k_s+c-1} & \dots & \binom{d}{k_s+c-1} \\[6pt]
    \end{bmatrix}.
\end{equation*}

 As stated previously, we have that $k_n+c-1 \geq 0$ and $d \geq k_n$ for all $n$. In particular, this implies $ \binom{d+c-1}{k_1+c-1} \geq 0$. It follows that the difference $(d+c-1) - (k_1+c-1) = d - k_1 \geq 0$. Then since $k_1 > k_2 > \dots > k_s$ and $k_n+c-1 \geq 0$ for all $n$, we also get $d - (n-1) - k_n \geq 0$, hence $\binom{d+c-n}{k_n+c-1} \geq 0$. From this we see that at least the entries on the main diagonal of the leftmost and uppermost maximal square submatrix of $B'$ are nonzero. Reversing the order of the rows and columns of $B'$, we get a matrix where its rightmost and lowermost maximal square submatrix is of the form in Corollary \ref{gvcor} and has a nonzero main diagonal. It follows from the Lindstr{\"o}m-Gessel-Viennot Lemma via Corollary \ref{gvcor} that $B'$ has a nonzero maximal minor, hence $B'$ and therefore $B$ has maximal rank. We have thus shown that $x+y$ is a strong Lefschetz element for $M$. 
\end{proof}

In Section \ref{appsect}, we will study codimension three $k$-algebras using modules of the form in the following corollary, which follows directly from the preceding theorem.

\begin{corollary}\label{cor}
    Let $I = (x^{\alpha}, y^{\beta})$ and $J = (x^a, y^b)$ be ideals in $S$ where $0 \leq \alpha \leq a$ and $0 \leq \beta \leq b$. Then $I/J$ has the SLP as an $S$-module. 
\end{corollary}

It is not necessarily true that similar modules in three variables have the SLP.

\begin{example} The $k[x,y,z]$-module $I/J$ where $I = (x^2,y^2,z^2)$ and $J = (x^3,y^3,z^3)$ does not have the strong Lefschetz property, although it does have a unimodal Hilbert series. Indeed, the map $\times (x+y+z) \colon [I/J]_3 \to [I/J]_4$ fails to be injective since $x^2(y-z)+y^2(z-x)+z^2(x-y)$ is in the kernel. But $\text{dim}_k [I/J]_3 = \text{dim}_k [I/J]_4$, so $I/J$ does not even have the weak Lefschetz property.  
\end{example}

Using the work of Lindsey (\cite{lindsams}), however, we can study the Lefschetz properties for $k[x,y,z]$-modules which are tensor product extensions of those in Corollary \ref{cor}. We provide below the relevant results and refer the reader to \cite{lindsams} or \cite{lindsey-thesis} for more details. 

\begin{definition}[\cite{lindsams}]
    Let $M = \bigoplus_{j=p}^q M_i$ be a graded Artinian module over $k[x_1,\dots,x_n]$ where $ 0 \leq q \leq p$ and $M_p$ and $M_q$ are both nonzero. We say that the Hilbert series of $M$ is \textit{almost centered} if 
    $h_{p+i-1} \leq h_{q-i} \leq h_{p+i}$ for all $1 \leq i \leq \lfloor \frac{q-p}{2} \rfloor$ or if $h_{q-i+1} \leq h_{p+i} \leq h_{q-i}$ for all  $1 \leq i \leq \lfloor \frac{q-p}{2} \rfloor$.
    For convenience we sometimes simply say that $M$ is almost centered if the above is true. 
\end{definition} 

Notice that if $M$ is almost centered, so is any shifted version of it.

\begin{lemma}[\cite{lindsams}]\label{ac-thm}
    Let $M$ be as above and suppose that $M$ has the strong Lefschetz property. Then $M \otimes_k k[z]/(z^c)$ has the strong Lefschetz property for all $c \geq 0$ if and only if the Hilbert series of $M$ is almost centered. Furthermore, if the Hilbert series of $M$ is almost centered, then the Hilbert series of $M \otimes_k k[z]/(z^c)$ is also almost centered. 
\end{lemma}

If the Hilbert series of $M$ is symmetric, it is also almost centered. It follows that the strong Lefschetz property in the narrow sense is preserved when tensoring with an algebra of the form $k[z]/(z^c)$. In more generality, Watanabe proved in Corollary 3.5 of \cite{wat-tensor} that the tensor product of any two $k$-algebras with the SLP in the narrow sense also has the SLP in the narrow sense. 

We may now utilize Lemma \ref{ac-thm} to study extensions of our modules of interest which will be used later in Section \ref{appsect}.

\begin{theorem}\label{tensoredmodules}
    Let $M = (x^{\alpha}, y^{\beta})/ (x^a, y^b,z^c)$ be a module over $S = k[x,y,z]$ where $0 \leq \alpha \leq a$ and $0 \leq \beta \leq b$. Then $M$ has the SLP for all $c \geq 0$ if  $\text{min}\{\alpha, \beta\} < \text{max}\{\alpha, \beta\} \leq 2$ or if $\text{min}\{\alpha, \beta\} < \text{max}\{\alpha, \beta\} = \text{min}\{\alpha+\beta, a,b\}$.
\end{theorem}
\begin{proof}
     Let $I/J$ be as in Corollary \ref{cor} and consider the $h_i = \text{dim}_k [I/J]_i$. By straightforward counting of monomials in each degree we get that 
     \begin{equation*}
     \begin{split}
     h_i &= \text{max}\{i-\alpha +1, 0 \} + \text{max}\{i-\beta+1,0\} - \text{max}\{i-\beta - \alpha +1,0\} \\  & \text{\ \ \ } -  \text{max}\{i-a+1, 0\} - \text{max}\{i-b+1,0\} +\text{max}\{i-a-b+1,0\}.
     \end{split}
     \end{equation*}
        Notice that the last term in the above expression is always zero; we have included it only for the sake of clarity. 
        
     Observe that the Hilbert series of $I/J$ is symmetric (hence also almost centered) if and only if $\text{min}\{\alpha, \beta\} \neq \text{max}\{\alpha, \beta\} = \text{min}\{\alpha+\beta, a,b\}$. Since $M \cong I/J \otimes_k k[z]/(z^c)$, $M$ has the SLP for all $c \geq 0$ by Lemma \ref{ac-thm} if $\text{min}\{\alpha, \beta\} \neq \text{max}\{\alpha, \beta\} = \text{min}\{\alpha+\beta, a,b\}$. 

    Next we consider the case in which the Hilbert series of $I/J$ is not necessarily symmetric. Assume that $\text{min}\{\alpha, \beta\} = \alpha$ and suppose that $\alpha < \beta \leq 2$. We may also assume without loss of generality that $a \leq b$ and that, to avoid triviality, that $\alpha \geq 1$ and $\beta \geq 1$. We will show that for the Hilbert series of $I/J$ we have $h_{\alpha-1+i} \leq h_{q-i} \leq h_{\alpha+i}$ for all $1 \leq i \leq \lfloor \frac{q-\alpha}{2} \rfloor$ where $q = a+b-2$, hence $I/J$ is almost centered.

     First we prove the inequality $h_{q-i} \leq h_{\alpha+i}$ for all $1 \leq i \leq \lfloor \frac{q-\alpha}{2} \rfloor$. Replacing $\alpha$ by $1$, per our assumptions, we have
     \begin{equation*}
     \begin{split}
     h_{1+i} - h_{q-i} &= 2i+\beta -2a-b +4 - \text{max}\{i-a+2, 0\} -\text{max}\{i-b+2,0 \} 
     \\& \text{\ \ \ } +\text{max}\{a+b-\beta -i -2,0\}+\text{max}\{a-i-1,0\}.
     \end{split}
     \end{equation*}
     We proceed by cases. First note that $\text{max}\{a+b-\beta -i -2,0\} = a+b-\beta -i -2 \geq 0$, since $i \leq \lfloor \frac{q-\alpha}{2} \rfloor$ implies $i < b-1$.
    
     Suppose that $i> a-1$, so then $\text{max}\{a-i-1,0\} = 0$ and $\text{max}\{i -a +2,0\}  = i-a+2$. Suppose also that $i < b-2$, so $\text{max}\{i-b+2,0\} = 0$. Then we get 
     \begin{equation*}
        \begin{split}
        h_{1+i} - h_{q-i} &=  2i+\beta -2a-b +4 - (i-a+2) +(a+b-\beta -i -2) +(a-i-1) = 0
         \end{split}
     \end{equation*}
     so the Hilbert series satisfies the almost centered condition for $i$ in the specified range. 
     
     On the other hand if $i \geq b-2$, we get 
    $h_{1+i} - h_{q-i} =  b-i-2 \geq b - (b-2) -2 \geq 0$
     using again the fact that $i < b-1$. 

    Next assume that $i \leq a-1$, so $\text{max}\{a-i-1,0\} = a-i-1$. We split our considerations into two main cases, $i < a-2$ and $i \geq a- 2$.
    If $i < a-2$, $\text{max}\{i-a+2,0\} = 0$ and $\text{max}\{i - b+2,0\} = 0$. Then 
    \begin{equation*}
     \begin{split}
     h_{1+i} - h_{q-i} &=  2i+\beta -2a-b +4 +(a+b-\beta -i -2) +(a-i-1) = 1 \\     
     \end{split}
     \end{equation*}
     which is of course greater than zero. In the other case, if $i \geq a - 2$, first suppose $i < b-2$. Then $ h_{1+i} - h_{q-i} = a-i -1 \geq a-(a-1) -1 = 0$, as desired If instead $i \geq b-2$, we have $ h_{1+i} - h_{q-i} = a+b -2i -3 \geq 0$, where the last inequality follows from the fact that  $i \leq \lfloor \frac{q-\alpha}{2} \rfloor$. 

     We must now prove that $h_i = h_{\alpha-1+i} \leq h_{q-i}$ for all $1\leq i \leq \lfloor \frac{q-\alpha}{2} \rfloor$. We have
     \begin{equation*}
     \begin{split}
     h_{q-i} - h_{i} &= 2a +b - 3i - 3 -\text{max}\{a+b -\beta -i-2, 0 \} - \text{max}\{a-i-1,0\}\\
     & \text{\ \ \ } +\text{max}\{i-\beta, 0\} + \text{max}\{i -a+1,0\} + \text{max}\{i -b+1, 0\}.
     \end{split}
     \end{equation*}
     Through a similar discussion, with our assumptions, we can prove that $h_{q-i} - h_{\alpha-1+i} \geq 0$ in all possible cases. Thus the Hilbert series of $I/J$ is almost centered, and the result follows via Lemma \ref{ac-thm}.        
\end{proof}

\begin{remark}\label{badtensorex}
 Modules of the form in Theorem \ref{tensoredmodules} need not be almost centered in general. For example, consider the $k[x,y]$-module $M = (x^2,y^2)/(x^4,y^4)$ which has the Hilbert series $2t^2+4t^3+3t^4+2t^5+t^6$. It is not hard to see that this Hilbert series is not almost centered. So, from Theorem \ref{ac-thm} we know that the $k[x,y,z]$-module $N = (x^2,y^2)/(x^4,y^4,z^c)$ must fail the SLP for at least one positive integer $c$. Indeed, if $c = 3$, $N$ does not have the SLP. The multiplication map $\times (x+y+z)^3 \colon N_3 \to N_6$ is not injective although $\text{dim}_k N_3 = \text{dim}_k N_7$; the element $(x-y)(x^2+y^2)$ is mapped to zero. 
\end{remark}

Theorem \ref{mainmodulethm}, though stated for monomial ideals $I$, is true for non-monomial ideals as well. To prove this we need the following lemma:

\begin{lemma}\label{initiallemma}
    Let $J \subset I$ be homogeneous ideals (not necessarily monomial) in $S = k[x_1, \dots, x_n]$ with $J$ Artinian. For any term order $<$, if $\text{in}_< (I)/\text{in}_< (J)$ has the WLP (respectively, SLP) as an $S$-module, then $I/J$ has the WLP (respectively, SLP) as an $S$-module.
\end{lemma}
    \begin{proof}
Assume for some term order $<$ that $\text{in}_< (I)/\text{in}_< (J)$ has the SLP. Let $\omega = (\omega_1, \dots, \omega_n) \in \mathbb{Z}_{\geq 0}^n$ be an integral weight such that $\text{in}_{\omega}(J) = \text{in}_<(J)$ and $\text{in}_{\omega}(J) = \text{in}_<(J)$. We will show that for every degree $d$ there is a general linear form $\ell$ such that the multiplication map $\times \ell^d \colon [I/J]_i \to [I/J]_{i+d}$ has maximal rank for all $i$. This is equivalent to showing that $I/J$ has the SLP: since $I/J$ is Artinian we need to check only finitely many degrees $d$, hence the intersection of the Zariski open sets corresponding to each of these $\ell$ is nonempty, being finite.

Now $I$ and $J$ have the same Hilbert functions as $\text{in}_{\omega}(I)$ and $\text{in}_{\omega}(J)$, respectively, and $\times \ell^d$ has maximal rank on $\text{in}_{\omega}(I)/\text{in}_{\omega}(J)$. So, denoting the Hilbert function by  $h_i(\textrm{--}) = \text{dim}_k [\textrm{--}]_i$,  it is enough to show that for all $i$ we have, 
\[
h_i(J+\ell \cdot I) \geq h_i(\text{in}_{\omega}(J) + \ell^d \cdot \text{in}_{\omega}(I)).
\]

Let $t$ be a new indeterminate and let $\widetilde{f}$ represent the homogenization of the polynomial $f \in S$ with respect to $\omega$, i.e., $\text{deg}(x_i) = \omega_i$ for all $i$, using $t$ with $\text{deg}(t) = 1$. Then let $\widetilde{I} = \{ \widetilde{f} | f \in I \}$ and $\widetilde{J}  = \{ \widetilde{f} | f \in J \}$ be the homogenizations of $I$ and $J$ in $S[t]$. Denote by $\mathbb{A}(S_1)$ the affine space of $S_1$. Fix $U \subseteq \mathbb{A}(S_1)$, a nonempty Zariski open set of linear forms which are SLE (respectively, WLE) for $\text{in}_< (I)/\text{in}_< (J)$. Fix also $V \subseteq \mathbb{A}(S_1)$ a nonempty Zariski open set of linear forms $\ell$ such that, for every $d$ and for every $i$,  $h_i(J+\ell^d \cdot I)$ is as large as possible. Note that $U \cap V \neq \emptyset$, and let $\ell \in U \cap V$. 

Denote by $I_{\alpha}$ and $J_{\alpha}$ the ideals of $S \cong S[t]/(t-\alpha)$ obtained by evaluating $t$ at some $\lambda \in k$. For a general choice of $\lambda$, $h_i(J_{\lambda}+\ell^d \cdot I_{\lambda}) \geq h_i(J_0 + \ell^d \cdot I_0)$. Notice that $J_0 = \text{in}_{\omega}(J)$ and $I_0 = \text{in}_{\omega}(I)$. Furthermore let $D_{\lambda} \in GL_n(k)$ be the matrix 
\[  D_{\lambda} =
  \bracketMatrixstack{
    \lambda^{-\omega_1} & 0 & 0\\
    0 & \cddots & 0 \\
    0 & 0 & \lambda^{-\omega_n}
  }
\]
and notice that $I_{\lambda} = D_{\lambda} I$ and $J_{\lambda} = D_{\lambda} J$. Thus
\begin{align*}
h_i(J_{\lambda}+\ell^d \cdot I_{\lambda}) &= h_i(D_{\lambda}(J + D_{\lambda}^{-1}(\ell^d)\cdot I)) \\ &= h_i(J+D_{\lambda}^{-1}(\ell^d)\cdot I) = h_i(J+(D_{\lambda}^{-1}(\ell))^d\cdot I)   \\ &\leq h_i(J+\ell^d \cdot I)
\end{align*}
where the last inequality follows from the fact that $\ell \in V$. Thus we have the result for the SLP. For the WLP, we merely need to replace $d$ by $1$ in the proof for the SLP and the result follows.
    \end{proof}

Although the above is stated for ideals, it is also true for $I$ and $J$ graded submodules of a free graded $S$-modules, even without the condition that the colength of $J$ is finite. The proof is similar; for more details on the strategy of such a proof see Section $2$ of \cite{caviglia-murai}.

\begin{theorem}
    Every homogeneous ideal $I$ (not necessarily monomial) in a Clements-Lindstr{\"o}m ring $A$ of embedding dimension two has the SLP as an $A$-module.
\end{theorem}
\begin{proof}
    A Clements-Lindstr{\"o}m ring of embedding dimension two is of the form $A = k[x,y]/(x^a,y^b)$ where $0 \leq a \leq \infty$ and $0 \leq b \leq \infty$. For this proof we may assume that $A$ is Artinian, i.e., $a$ and $b$ are both finite, since otherwise $x$ or $y$ is a non zero-divisor on $I/(x^a,y^b)$ and then the SLP for $I/(x^a,y^b)$ as a $k[x,y]$-module (which is equivalent to the SLP for $I$ as an $A$-module) is trivial. The result then follows from Lemma \ref{initiallemma} and Theorem \ref{mainmodulethm}.
\end{proof}

\section{Monomial algebras of type two}\label{appsect}

Monomial Artinian level algebras were examined in the monograph \cite{monograph}, a main results of which was that such an algebra of type two in three variables has the WLP in characteristic zero. However, the WLP can fail if the algebra is not level. The non-level case has been studied by Cook and Nagel in \cite{cooknag}, in which they identified the only two possible forms for type two Artinian monomial algebras of codimension three, where $S = k[x,y,z]$: 
\begin{enumerate}
    \item $S/I$ where $I = (x^a, y^b, z^c, x^{\alpha} z^{\gamma})$ where $0< \alpha < a$ and $0< \gamma < c$;
    \item $S/I$ where $I = (x^a, y^b, z^c, x^{\alpha} z^{\gamma}, y^{\beta} z^{\gamma})$.
\end{enumerate}
Cook and Nagel gave a complete classification of the algebras of the above forms that have the weak Lefschetz property in characteristic zero in \cite{cooknag} using lattice paths, lozenge tilings, and perfect matchings. Their combinatorial method for studying the weak Lefschetz property cannot also be used to study the strong Lefschetz property, however, and so the SLP for these algebras is still open. 

Almost complete intersections, such as algebras of the first form, have been studied by many authors (\cite{seceleanu}, \cite{aci1}, \cite{aci2}, and \cite{aci4}, for example). Here we focus on the strong Lefschetz property for algebras of the second form. Using Corollary \ref{cor} and the method of central simple modules of Harima and Watanabe, we can prove the SLP for some cases of such algebras. 

\begin{theorem}
    Let $I = (x^a, y^b, z^c, x^{\alpha}z^{\gamma}, y^{\beta}z^{\gamma})$ be an Artinian monomial ideal in $S = k[x,y,z]$ where $0 < \alpha < a$, $0<\beta < b$, and $0< \gamma < c$. Then $S/I$ has the strong Lefschetz property if any of the following conditions hold:
    \begin{enumerate}
    \item $\alpha+\beta -1 \leq a+b -c \leq \alpha+\beta +1$;
    \item $min \{\alpha,\beta \} \neq max \{\alpha,\beta \} = min\{\alpha+\beta, a, b\}$ and \\ $max \{\alpha,\beta \} -\gamma -1 \leq a+b-c \leq max \{\alpha,\beta \} -\gamma +1$;
    \item $min \{\alpha,\beta \} < max \{\alpha,\beta \} \leq 2$ and $a+b+\gamma \leq c+2$.
    \end{enumerate}

\end{theorem}
\begin{proof} 
     We first consider the central simple modules $V_{i,x}$ of $A$ with respect to the linear form $x$ (or using the linear form $y$, which yields the same result). We get only two central simple modules:
    \begin{align*}
    V_{1,x} &= \frac{(0:x^a)+(x)}{(0:x^{a-1})+(x)} \cong \frac{k[y,z]}{(y^b,z^{\gamma})} \\
    V_{2,x} &= \frac{(0:x^{\alpha})+(x)}{(0:x^{\alpha-1})+(x)} \cong \frac{k[y,z]}{(y^{\beta}, z^{c-\gamma})}(-\gamma).
    \end{align*}
    Both of these modules, being monomial complete intersections, have the SLP in the narrow sense. We also have $f_1 = a$ and $f_2 = \alpha$, so $\widetilde{V}_{1,x}$ and $\widetilde{V}_{2,x}$ have symmetric Hilbert series with reflecting degrees 
    \begin{align*}
        r_{1,x} &= \frac{a+b+\gamma -3}{2} \\
        r_{2,x} &= \frac{\gamma + \alpha +\beta +c -3}{2} 
    \end{align*}
    which coincide if and only if 
    \[\alpha+\beta -1 \leq a+b -c \leq \alpha+\beta +1.\]
    It follows from Theorem \ref{csm-nonsym} and Remark \ref{directsumlp} that $S/I$ has the SLP if condition $(1)$ holds.

    We next consider the central simple modules $V_{i,z}$ of $S/I$ with respect to the linear form $z$, of which there are again only two: 
    \begin{align*}
        V_{1,z} = \frac{(0:z^c)+(z)}{(0:z^{c-1})+(z)} &\cong \frac{k[x,y]}{(x^{\alpha}, y^{\beta})} \\
        V_{2,z} = \frac{(0:z^{\gamma})+(z)}{(0:z^{\gamma -1})+(z)} &\cong \frac{(x^{\alpha}, y^{\beta})}{(x^a,y^b)} \text{ as a quotient of ideals in } k[x,y].
    \end{align*}    
   $V_{1,z}$ has both the SLP and a symmetric Hilbert series. We also have that \[\widetilde{V}_{1,z} = V_{1,z} \otimes k[t]/(t^c) \cong k[x,y,t]/(x^{\alpha}, y^{\beta}, t^c)\] 
    which has a symmetric Hilbert series that is nonzero starting in degree 0 until degree $\alpha +\beta +c -3$. 
    By Corollary \ref{cor}, $V_{2,z}$ also has the SLP. Taking the tensor product with $k[t]/(t^{\gamma})$, we see that 
    \[\widetilde{V}_{2,z} \cong \frac{(x^{\alpha}, y^{\beta})}{(x^a, y^b, t^{\gamma})} \text{ as a quotient of ideals in } k[x,y,t]\] 
    and $V_{2,z}$ is nonzero in degrees $\text{min}\{\alpha,\beta\}$ to $a+b+\gamma-3$. Its Hilbert series is symmetric if and only if $\text{min}\{\alpha, \beta\} \neq \text{max}\{\alpha, \beta\} = \text{min}\{\alpha+ \beta, a, b \}$, hence the same is true for $\widetilde{V}_{2,z}$. When this condition for symmetry is met, we can compare the reflecting degrees $r_{1,z}$ and $r_{2,z}$ of $V_{1,z}$ and $V_{2,z}$, respectively. We get
    \begin{align*}
    r_{1,z} &= \frac{\alpha +\beta +c -3}{2} \\
    r_{2,z} &= \frac{\text{min}\{\alpha,\beta\} + a+b+\gamma -3}{2}
    \end{align*}
    so $r_{1,z}$ and $r_{2,z}$ coincide if and only if $\text{max}\{\alpha, \beta\} -\gamma -1 \leq a+b-c \leq \text{max}\{\alpha, \beta\} - \gamma +1$. So $S/I$ has the SLP if condition $(2)$ holds.   

    It is tedious to determine precisely what the Hilbert series of $\widetilde{V}_{2,z}$ looks like when it is not symmetric, but we can still find some conditions which force $M$ to have the strong Lefschetz property. 

    Suppose that $c \geq \text{max}\{\alpha, \beta\}$. Then in all degrees $\alpha+\beta-2$ to $c-1$ or in all degrees $c-1$ to $\alpha+\beta-2$, depending on which is larger, $H_{\widetilde{V}_{1,z}}(t)$ is equal to its maximum value . First consider the case where $\alpha+\beta-2 \leq c-1$. For any i and d such that $\alpha+\beta-2 \leq i$ and $i+d \leq c-1$, then, $\times \ell^d \colon [\widetilde{V}_{1,z}]_i \to [\widetilde{V}_{1,z}]_{i+d}$ is bijective. Thus if the Hilbert series of $\widetilde{V}_{2,z}$ is nonzero only in degrees where $\widetilde{V}_{1,z}$ has a peak, $\ell$ is a strong Lefschetz element for $\widetilde{V} = \widetilde{V}_{1,z} \oplus \widetilde{V}_{2,z}$ (see Remark \ref{directsumlp}). This happens if and only if $\alpha +\beta -2 \leq \text{min}\{\alpha, \beta\}$, i.e., $\text{max}\{\alpha,\beta\} \leq 2$, and $a+b+\gamma -3\leq c-1$. By Theorem \ref{tensoredmodules} and Remark \ref{badtensorex}, we know that $\widetilde{V}_{1,z}$ may fail the SLP unless $\text{min}\{\alpha,\beta \} =1$. So we see that if condition $(3)$ of the proposition holds, $S/I$ has the SLP by Theorem \ref{csm-nonsym}. 
    
    On the other hand, if $c-1 \leq \alpha+\beta-2$, by a similar argument, $S/I$ has the SLP if $c-1 \leq \text{min}\{\alpha,\beta\}$ and $a+b+\gamma \leq \alpha+\beta+1$. The latter condition is impossible given the assumptions on the degrees of the generators of $I$.
\end{proof}

\bibliography{refm}
\bibliographystyle{acm}
\end{document}